\documentclass[12pt]{article}
\usepackage{amsmath,amscd,amsthm,a4,amssymb}

\def\Rn{{\mathbb{R}^n}}

\def\a {\alpha}
\def\i{\infty}

\def\Lploc{L_p^{\rm loc}(\Rn)}
\def\L1loc{L_1^{\rm loc}(\Rn)}

\def\dual{\,^{^{\complement}}\!}

\newcommand{\ess}{\mathop{\rm ess \; sup}\limits}
\newcommand{\es}{\mathop{\rm ess \; inf}\limits}

\newtheorem{thm}{Theorem}[section]
 \newtheorem{cor}[thm]{Corollary}
 \newtheorem{lem}[thm]{Lemma}
 
 \theoremstyle{definition}
 \newtheorem{defn}[thm]{Definition}
 \theoremstyle{remark}
 \newtheorem{rem}[thm]{Remark}
 
 \numberwithin{equation}{section}

\begin{document}

\begin{center}
\Large \bf Generalized local Morrey spaces and fractional integral operators with rough kernel
\end{center}

\centerline{\large Vagif S. Guliyev$^{a,b,}$\footnote{
{The research of V. Guliyev was partially supported by the grant of Science Development Foundation under the President of the Republic of Azerbaijan  project EIF-2010-1(1)-40/06-1
and by the Scientific and Technological Research Council of Turkey  (TUBITAK Project No: 110T695) and by the grant of 2010-Ahi Evran University Scientific Research Projects (PYO-FEN 4001.12.18).}
\\
E-mail adresses: vagif@guliyev.com (V.S. Guliyev).}}

\

\centerline{$^{a}$\it Department of Mathematics, Ahi Evran University, Kirsehir, Turkey}

\centerline{$^{b}$\it Institute of Mathematics and Mechanics of NAS of Azerbaijan, Baku}

\

\begin{abstract}
Let $M_{\Omega,\a}$ and $I_{\Omega,\a}$ be the fractional maximal and integral operators with rough kernels, where $0 < \a < n$. In this paper, we shall study the continuity properties of $M_{\Omega,\a}$ and $I_{\Omega,\a}$ on the generalized local Morrey spaces $LM_{p,\varphi}^{\{x_0\}}$. The boundedness of their commutators with local Campanato functions is also obtained.
\end{abstract}

\



\

\section{Introduction}

For $x \in \Rn $ and $r > 0$, let  $B(x,r)$ denote the open ball centered at $x$ of radius $r$ and $|B(x,r)|$ is the Lebesgue measure of the ball $B(x,r)$.
Let $\Omega\in L^{s}(S^{n-1})$ be homogeneous of degree zero on $\Rn$, where $S^{n-1}$ denotes
the unit sphere of $\Rn$ $(n \ge 2)$ equipped with the normalized Lebesgue measure
$ d\sigma$ and $s > 1$. For any $0 < \a < n$, then the fractional integral operator with rough kernel
$I_{\Omega,\a}$ is defined by
$$
I_{\Omega,\a}f(x) = \int_{\Rn }\frac{\Omega(x-y)}{|x-y|^{n-\a}} f(y) dy
$$
and a related fractional maximal operator with rough kernel $M_{\Omega,\a}$ is defined by
$$
M_{\Omega,\alpha}f(x)=\sup_{t>0} |B(x,t)|^{-1+ \frac{\alpha}{n}} \int_{B(x,t)} |\Omega(x-y)| \,|f(y)|dy.
$$

If $\a=0$, then $M_{\Omega} \equiv M_{\Omega,0}$ is the Hardy-Littlewood maximal operator with rough kernel.
It is obvious that when $\Omega \equiv 1$, $I_{\Omega,\a}$ is the Riesz potential $I_{\a}$ and $M_{\Omega,\a}$ is the maximal operator $M_{\a}$.


\vspace{2mm}
{\bf Theorem A}  {\it Suppose that $\Omega \in L_s(S^{n-1})$, $ 1<s\leq\infty$, be a homogeneous of degree zero. Let $0 < \a < n$, $1 \le p < \frac{n}{\a}$, and $\frac{1}{q}=\frac{1}{p}-\frac{\a}{n}$. If $s'< p$ or $q< s$, then the operators $M_{\Omega,\a}$ and $I_{\Omega,\a}$ are bounded bounded from $L_p(\Rn)$ to $L_q(\Rn)$.}

\vspace{3mm}

Let $b$ be a locally integrable function on $\Rn$, then for $0 < \a < n$, we shall
define the commutators generated by fractional maximal and integral operators
with rough kernels and $b$ as follows.
\begin{align*}
M_{\Omega,b,\a}(f)(x) & = \sup_{t>0}|B(x,t)|^{-1+\frac{\alpha}{n}}\int_{B(x,t)} |b(x)-b(y)| |f(y)| |\Omega(x-y)|dy,
\\
[b, I_{\Omega,\a}]f(x) & = b(x) I_{\Omega,\a}f_{1}(x) - I_{\Omega,\a}(bf)(x)
\\
& = \int_{\Rn }\frac{\Omega(x-y)}{|x-y|^{n-\a}} [b(x)-b(y)] f(y) dy.
\end{align*}

\vspace{2mm}
{\bf Theorem B} {\it Suppose that $\Omega \in L_s(S^{n-1})$, $ 1<s\leq\infty$, be a homogeneous of degree zero. Let $0 < \a < n$, $1 \le p < \frac{n}{\a}$, $\frac{1}{q}=\frac{1}{p}-\frac{\a}{n}$ and $b \in BMO(\Rn)$. If $s'< p$ or $q< s$, then the operators $M_{\Omega,b,\a}$ and $[b, I_{\Omega,\a}]$ are bounded from $L_p(\Rn)$ to $L_q(\Rn)$.}

\vspace{3mm}

The classical Morrey spaces $M_{p,\lambda}$ were first introduced by Morrey in \cite{Morrey} to
study the local behavior of solutions to second order elliptic partial differential
equations. For the boundedness of the Hardy-Littlewood maximal operator,
the fractional integral operator and the Calder\'{o}n-Zygmund singular integral
operator on these spaces, we refer the readers to \cite{Adams, ChFra, Peetre}. For the properties
and applications of classical Morrey spaces, see \cite{ChFraL1, ChFraL2, FazRag2, FazPalRag} and references therein.

In the paper, we prove the boundedness of the operators $I_{\Omega,\a}$ from one
generalized local Morrey space $LM_{p,\varphi_1}^{\{x_0\}}$ to
$LM_{q,\varphi_2}^{\{x_0\}}$, $1<p< q<\infty$, $1/p-1/q=\a/n$, and from the space
$LM_{1,\varphi_1}^{\{x_0\}}$ to the weak space $WLM_{q,\varphi_2}^{\{x_0\}}$, $1< q<\infty$, $1-1/q=\a/n$.
In the case $b \in CBMO_{p_2}$, we find the sufficient
conditions on the pair $(\varphi_1,\varphi_2)$ which
ensures the boundedness of the commutator operators $[b, I_{\Omega,\a}]$ from
$LM_{p_1,\varphi_1}^{\{x_0\}}$ to $LM_{q,\varphi_2}^{\{x_0\}}$, $1< p <\infty$, $\frac{1}{p}=\frac{1}{p_1}+\frac{1}{p_2}$, $\frac{1}{q}=\frac{1}{p}-\frac{\a}{n}$, $\frac{1}{q_1}=\frac{1}{p_1}-\frac{\a}{n}$.

By $A \lesssim B$ we mean that $A \le C B$ with some positive constant $C$
independent of appropriate quantities. If $A \lesssim B$ and $B \lesssim A$, we
write $A\approx B$ and say that $A$ and $B$ are  equivalent.


%
%

\

\section{Generalized local Morrey spaces }

We find it convenient to define the generalized Morrey spaces in
the form as follows.
\begin{defn}
Let $\varphi(x,r)$ be a positive measurable function on $\Rn \times (0,\infty)$ and $1 \le p < \infty$. We denote by
$M_{p,\varphi} \equiv M_{p,\varphi}(\Rn)$ the generalized Morrey space, the space of all
functions $f\in L_p^{\rm loc}(\Rn)$ with finite quasinorm
$$
\|f\|_{M_{p,\varphi}} = \sup\limits_{x\in\Rn, r>0} \varphi(x,r)^{-1} \, |B(x,r)|^{-\frac{1}{p}} \, \|f\|_{L_p(B(x,r))}.
$$
Also by $WM_{p,\varphi} \equiv WM_{p,\varphi}(\Rn)$ we denote the weak
generalized Morrey space of all functions $f\in WL_p^{\rm loc}(\Rn)$ for which
$$
\|f\|_{WM_{p,\varphi}} = \sup\limits_{x\in\Rn, r>0} \varphi(x,r)^{-1} \,
|B(x,r)|^{-\frac{1}{p}} \, \|f\|_{WL_p(B(x,r))} < \infty.
$$
\end{defn}

According to this definition, we recover the Morrey space $M_{p,\lambda}$ and weak Morrey space
$WM_{p,\lambda}$ under the choice $\varphi(x,r)=r^{\frac{\lambda-n}{p}}$:
$$
M_{p,\lambda}=
M_{p,\varphi}\Big|_{\varphi(x,r)=r^{\frac{\lambda-n}{p}}},
~~~~~~~~
WM_{p,\lambda}=
WM_{p,\varphi}\Big|_{\varphi(x,r)=r^{\frac{\lambda-n}{p}}}.
$$

\begin{defn}
Let $\varphi(x,r)$ be a positive measurable function on $\Rn \times (0,\infty)$ and $1 \le p < \infty$. We denote by $LM_{p,\varphi} \equiv LM_{p,\varphi}(\Rn)$
the generalized local Morrey space, the space of all functions $f\in L_p^{\rm loc}(\Rn)$ with finite quasinorm
$$
\|f\|_{LM_{p,\varphi}} = \sup\limits_{r>0}
\varphi(0,r)^{-1} \, |B(0,r)|^{-\frac{1}{p}} \, \|f\|_{L_p(B(0,r))}.
$$
Also by $WLM_{p,\varphi} \equiv WLM_{p,\varphi}(\Rn)$ we denote the weak generalized Morrey space of all functions $f\in WL_p^{\rm loc}(\Rn)$ for which
$$
\|f\|_{WLM_{p,\varphi}}  = \sup\limits_{r>0} \varphi(0,r)^{-1} \, |B(0,r)|^{-\frac{1}{p}} \, \|f\|_{WL_p(B(0,r))} < \infty.
$$
\end{defn}

\begin{defn}
Let $\varphi(x,r)$ be a positive measurable function on $\Rn \times (0,\infty)$ and $1 \le p < \infty$. For any fixed $x_0 \in \Rn$ we denote by
$LM_{p,\varphi}^{\{x_0\}} \equiv LM_{p,\varphi}^{\{x_0\}}(\Rn)$ the generalized local Morrey space, the space of all
functions $f\in L_p^{\rm loc}(\Rn)$ with finite quasinorm
$$
\|f\|_{LM_{p,\varphi}^{\{x_0\}}} = \|f(x_0+\cdot)\|_{LM_{p,\varphi}}.
$$
Also by $WLM_{p,\varphi}^{\{x_0\}} \equiv WLM_{p,\varphi}^{\{x_0\}}(\Rn)$ we denote the weak
generalized Morrey space of all functions $f\in WL_p^{\rm loc}(\Rn)$ for which
$$
\|f\|_{WLM_{p,\varphi}^{\{x_0\}}}  = \|f(x_0+\cdot)\|_{WLM_{p,\varphi}} < \infty.
$$
\end{defn}

According to this definition, we recover the local Morrey space $LM_{p,\lambda}^{\{x_0\}}$ and
weak local Morrey space $WLM_{p,\lambda}^{\{x_0\}}$ under the choice
$\varphi(x_0,r)=r^{\frac{\lambda-n}{p}}$:
$$
LM_{p,\lambda}^{\{x_0\}}=LM_{p,\varphi}^{\{x_0\}}\Big|_{\varphi(x_0,r)=r^{\frac{\lambda-n}{p}}},
~~~~~~
WLM_{p,\lambda}^{\{x_0\}}=WLM_{p,\varphi}^{\{x_0\}}\Big|_{\varphi(x_0,r)=r^{\frac{\lambda-n}{p}}}.
$$

Wiener \cite{Wiener1, Wiener2} looked for a way to describe the behavior of a function at the infinity. The
conditions he considered are related to appropriate weighted $L_q$ spaces. Beurling \cite{Beurl} extended
this idea and defined a pair of dual Banach spaces $A_{q}$ and $B_{q'}$, where $1/q + 1/q' = 1$. To be
precise, $A_q$ is a Banach algebra with respect to the convolution, expressed as a union of certain
weighted $L_q$ spaces; the space $B_{q'}$ is expressed as the intersection of the corresponding weighted $L_{q'}$ spaces.
Feichtinger \cite{Feicht} observed that the space $B_q$ can be described by
\begin{equation} \label{ank01}
\left\| f\right\|_{B_{q}} = \sup_{k \ge 0} 2^{-\frac{kn}{q}} \|f \chi_{k}\|_{L_{q}(\Rn)},
\end{equation}
where $\chi_{0}$ is the characteristic function of the unit ball $\{x \in \Rn: |x| \le 1 \}$, $\chi_{k}$ is the characteristic
function of the annulus $\{x \in \Rn : 2^{k-1} < |x| \le 2^{k}\}$, $k = 1, 2, \ldots$.
By duality, the space $A_{q}(\Rn)$, called Beurling algebra now, can be described by
\begin{equation} \label{ank02}
\left\| f\right\|_{A_{q}} = \sum\limits_{k=0}^{\i} 2^{-\frac{kn}{q'}} \|f \chi_{k}\|_{L_{q}(\Rn)}.
\end{equation}

Let $\dot{B}_{q}(\Rn)$ and $\dot{A}_{q}(\Rn)$ be the homogeneous versions of $B_{q}(\Rn)$ and $A_{q}(\Rn)$ by taking $k \in \mathbb{Z}$ in \eqref{ank01} and
\eqref{ank02} instead of $k \ge 0$ there.


If $\lambda<0$ or $\lambda>n$, then $LM_{p,\lambda}^{\{x_0\}}(\Rn)={\Theta}$, where $\Theta$ is the set of all functions equivalent to $0$ on $\Rn$.
Note that $LM_{p,0}(\Rn)=L_{p}(\Rn)$ and $LM_{p,n}(\Rn)=\dot{B}_{p}(\Rn)$.
$$
\dot{B}_{p,\mu}=LM_{p,\varphi}\Big|_{\varphi(0,r)=r^{\mu n}},
~~~~~~
W\dot{B}_{p,\mu}=WLM_{p,\varphi}\Big|_{\varphi(0,r)=r^{\mu n}}.
$$

Alvarez, Guzman-Partida and Lakey \cite{AlvLanLakey}
in order to study the relationship between central $BMO$ spaces and Morrey spaces, they introduced $\lambda$-central bounded mean oscillation spaces and central Morrey spaces
$\dot{B}_{p,\mu}(\Rn) \equiv LM_{p,n+np\mu}(\Rn)$, $\mu \in [-\frac{1}{p},0]$. If $\mu<-\frac{1}{p}$ or $\mu>0$, then $\dot{B}_{p,\mu}(\Rn)={\Theta}$. Note that $\dot{B}_{p,-\frac{1}{p}}(\Rn)=L_{p}(\Rn)$ and $\dot{B}_{p,0}(\Rn)=\dot{B}_{p}(\Rn)$. Also define the weak central Morrey spaces $W\dot{B}_{p,\mu}(\Rn) \equiv WLM_{p,n+np\mu}(\Rn)$.

Inspired by this, we consider the boundedness of fractional integral operator with rough kernel on generalized
local Morrey spaces and give the central bounded mean oscillation estimates for their commutators.


\

\section{ Fractional integral operator with rough kernels in the spaces $LM_{p,\varphi}^{\{x_0\}}$}

In this section we are going to use the following statement on the boundedness
of the weighted Hardy operator
$$
H^{\ast}_{w} g(t):=\int_t^{\infty} g(s) w(s) ds,~ 0<t<\infty,
$$
where $w$ is a fixed function non-negative and measurable on $(0,\i)$.
\begin{thm}\label{thm3.2.}
Let $v_1$, $v_2$ and $w$ be positive almost everywhere and measurable functions on $(0,\i)$. The inequality
\begin{equation} \label{vav01}
\ess_{t>0} v_2(t) H^{\ast}_{w} g(t) \leq C \ess_{t>0} v_1(t) g(t)
\end{equation}
holds for some $C>0$ for all non-negative and non-decreasing $g$ on $(0,\i)$ if and
only if
\begin{equation} \label{vav02}
B:= \ess_{t>0} v_2(t)\int_t^{\infty} \frac{w(s) ds}{\ess_{s<\tau<\infty} v_1(\tau)}<\i.
\end{equation}

Moreover, if $C^{\ast}$ is the minimal value of $C$ in \eqref{vav01}, then $C^{\ast}=B$.
\end{thm}
\begin{proof}
{\it Sufficiency.} Assume that \eqref{vav02} holds. Whenever $F$, $G$ are non-negative
functions on $(0,\i)$ and $F$ is non-decreasing, then
\begin{equation} \label{vav03}
\ess_{t>0} F(t) G(t) = \ess_{t>0} F(t) \ess_{s>t} G(s), ~~~ t>0.
\end{equation}
By \eqref{vav03} we have
\begin{align*}
\ess_{t>0} v_2(t) H^{\ast}_{w} g(t) & = \ess_{t>0} v_2(t) \int_t^{\infty} g(s) w(s) \,
\frac{\ess_{s<\tau<\i} v_1(\tau)}{\ess_{s<\tau<\i} v_1(\tau)} \, ds
\\
& \le \ess_{t>0} v_2(t) \int_t^{\infty} \frac{w(s) ds}{\ess_{s<\tau<\i} v_1(\tau)} \, \ess_{t>0} g(t)
\, \ess_{t<\tau<\i} v_1(\tau)
\\
& = \ess_{t>0} v_2(t) \int_t^{\infty} \frac{w(s) ds}{\ess_{s<\tau<\i} v_1(\tau)} \, \ess_{t>0} g(t) v_1(t)
\\
& \le B \, \ess_{t>0} g(t) v_1(t).
\end{align*}

{\it Necessity.} Assume that the inequality \eqref{vav01} holds. The function
\begin{equation*}
g(t) = \frac{1}{\ess_{t<\tau<\i} v_1(\tau)}, ~~ t>0
\end{equation*}
is nonnegative and non-decreasing on $(0,\i)$. Thus
\begin{equation*}
B=\ess_{t>0} v_2(t) \int_t^{\infty} \frac{w(s) ds}{\ess_{s<\tau<\i} v_1(\tau)}
\le C \ess_{t>0} \frac{v_1(t)}{\ess_{t<\tau<\i} v_1(\tau)} \le C,
\end{equation*}
hence $C^{\ast}=B$.
\end{proof}

In \cite{DingYZ} the following statements was proved by fractional integral operator with rough kernels $I_{\Omega,\a}$, containing the result in \cite{Miz, Nakai}.
\begin{thm}  \label{nakaiPot}
Suppose that $\Omega \in L_s(S^{n-1})$, $ 1<s\leq\infty$, be a homogeneous of degree zero. Let $0 < \a < n$, $1 \le s' < p < \frac{n}{\a}$, $\frac{1}{q}=\frac{1}{p}-\frac{\a}{n}$ and $\varphi(x,r)$ satisfy
 conditions
 \begin{equation} \label{nakcond}
 c^{-1} \varphi(x,r)\le \varphi(x,t)\le c \, \varphi(x,r)
 \end{equation}
whenever $r \le  t \le 2r$, where $c~(\ge 1)$ does not depend on $t$, $r$, $x\in \Rn$ and
\begin{equation} \label{MizNPot}
 \int_r^\infty t^{\a p} \varphi(x,t)^p \frac{dt}{t} \le C \,r^{\a p} \varphi(x,r)^p,
\end{equation}
where $C$ does not depend on $x$ and $r$.
Then the operators $M_{\Omega,\a}$ and $I_{\Omega,\a}$ are bounded from $M_{p,\varphi}$ to $M_{q,\varphi}$.
\end{thm}

The following statements, containing results obtained in \cite{Miz},
\cite{Nakai} was proved in \cite{GulDoc, GulJIA} (see also \cite{BurGulHus1}-\cite{BurGogGulMus2}, \cite{GulBook, GULAKShIEOT2012}).
\begin{thm}\label{nakaiVagifPot0}
Let $0 < \a < n$, $1 \le p < \frac{n}{\a}$, $\frac{1}{q}=\frac{1}{p}-\frac{\a}{n}$
and $(\varphi_1,\varphi_2)$ satisfy the condition
\begin{equation}\label{GulZPot}
\int_r^{\infty} t^{\a-1} \varphi_1(0,t) dt \le  C \, \varphi_2(0,r),
\end{equation}
where $C$ does not depend on $r$. Then the operators $M_{\a}$ and $I_{\a}$ are
bounded  from $LM_{p,\varphi_1}$ to $LM_{q,\varphi_2}$ for $p > 1$ and
from $LM_{1,\varphi_1}$ to $WLM_{q,\varphi_2}$ for $p=1$.
\end{thm}

\begin{lem}\label{lem3.3.Pot}
Suppose that $x_0 \in \Rn$, $\Omega \in L_s(S^{n-1})$, $ 1<s\leq\infty$, be a homogeneous of degree zero.
Let $0 < \a < n$, $1 \le p < \frac{n}{\a}$, and $\frac{1}{q}=\frac{1}{p}-\frac{\a}{n}$.
Then, for $p>1$ and $s' \le p$ or $q < s$ the inequality
\begin{equation*}\label{eq3.5.}
\|I_{\Omega,\a} f\|_{L_q(B(x_0,r))} \lesssim r^{\frac{n}{q}} \int_{2r}^{\i} t^{-\frac{n}{q}-1} \|f\|_{L_p(B(x_0,t))}dt
\end{equation*}
holds for any ball $B(x_0,r)$ and for all $f\in\Lploc$.


Moreover, for $p=1<q<s$ the inequality
\begin{equation}\label{eq3.5.WXM0}
\|I_{\Omega,\a} f\|_{WL_q(B(x_0,r))} \lesssim r^{\frac{n}{q}} \int_{2r}^{\i} t^{-\frac{n}{q}-1} \|f\|_{L_1(B(x_0,t))}dt,
\end{equation}
holds for any ball $B(x_0,r)$ and for all $f\in\L1loc$.

\end{lem}
\begin{proof}
Let $0<\a<n$, $1\le s'\le p<\frac{n}{\a}$ and $\frac{1}{q}=\frac{1}{p}-\frac{\a}{n}$. Set $B=B(x_0,r)$ for the ball centered at $x_0$ and of radius $r$.
We represent  $f$ as
\begin{equation}\label{repr}
f=f_1+f_2, \ \quad f_1(y)=f(y)\chi _{2B}(y),\quad
 f_2(y)=f(y)\chi_{\dual {(2B)}}(y), \ \quad r>0,
\end{equation}
and have
$$
\|I_{\Omega,\a} f\|_{L_{q}(B)} \le \|I_{\Omega,\a} f_1\|_{L_{q}(B)} +\|I_{\Omega,\a} f_2\|_{L_{q}(B)}.
$$

Since $f_1\in L_p(\Rn)$, $I_{\Omega,\a} f_1\in L_q(\Rn)$  and from the
boundedness of $I_{\Omega,\a}$ from  $L_p(\Rn)$ to $L_q(\Rn)$ it follows that:
\begin{equation*}
\|I_{\Omega,\a} f_1\|_{L_q(B)}\leq \|I_{\Omega,\a} f_1\|_{L_q(\Rn)}\leq
C\|f_1\|_{L_p(\Rn)}=C\|f\|_{L_p(2B)},
\end{equation*}
where constant $C >0$ is independent of $f$.

It's clear that $x\in B$, $y\in \dual {(2B)}$ implies
$\frac{1}{2}|x_0-y|\le|x-y|\le \frac{3}{2}|x_0-y|$. We get
$$
|I_{\Omega,\a} f_2(x)|\leq 2^{n-\a} c_1 \, \int_{\dual {(2B)}}\frac{|f(y)| |\Omega(x-y)|}{|x_0-y|^{n-\a}}dy.
$$
By Fubini's theorem we have
\begin{equation*}
\begin{split}
\int_{\dual {(2B)}}\frac{|f(y)| |\Omega(x-y)|}{|x_0-y|^{n-\a}}dy&\thickapprox
\int_{\dual {(2B)}}|f(y)| |\Omega(x-y)| \int_{|x_0-y|}^{\i}\frac{dt}{t^{n+1-\a}}dy
\\
&\thickapprox \int_{2r}^{\i}\int_{2r\leq |x_0-y|\le t} |f(y)| |\Omega(x-y)| dy \frac{dt}{t^{n+1-\a}}
\\
&\lesssim \int_{2r}^{\i}\int_{B(x_0,t) }|f(y)| |\Omega(x-y)| dy\frac{dt}{t^{n+1-\a}}.
\end{split}
\end{equation*}
Applying H\"older's inequality, we get
\begin{equation} \label{sal00Pot}
\begin{split}
&\int_{\dual {(2B)}}\frac{|f(y)| |\Omega(x-y)|}{|x_0-y|^{n-\a}}dy
\\
& \lesssim
\int_{2r}^{\i} \|f\|_{L_p(B(x_0,t))} \, \|\Omega(\cdot-y)\|_{L_s(B(x_0,r))}
\, |B(x_0,t)|^{1-\frac{1}{p}-\frac{1}{s}}  \, \frac{dt}{t^{n+1-\a}}
\\
& \lesssim \int_{2r}^{\i} \|f\|_{L_p(B(x_0,t))} \, \frac{dt}{t^{\frac{n}{q}+1}}.
\end{split}
\end{equation}

Moreover, for all $p\in [1,\i)$ the inequality
\begin{equation} \label{ves2Pot}
\|I_{\Omega,\a} f_2\|_{L_q(B)}\lesssim
r^{\frac{n}{q}}\int_{2r}^{\i}\|f\|_{L_p(B(x_0,t))}\frac{dt}{t^{\frac{n}{q}+1}}.
\end{equation}
is valid. Thus
\begin{equation*}
\|I_{\Omega,\a} f\|_{L_q(B)}\lesssim \|f\|_{L_p(2B)}+
r^{\frac{n}{q}}\int_{2r}^{\i}\|f\|_{L_p(B(x_0,t))}\frac{dt}{t^{\frac{n}{q}+1}}.
\end{equation*}

On the other hand,
\begin{align}  \label{sal01Pot}
\|f\|_{L_p(2B)} & \thickapprox \; r^{\frac{n}{q}} \|f\|_{L_p(2B)} \int_{2r}^{\i}\frac{dt}{t^{\frac{n}{q}+1}} \notag
\\
& \le  \;  r^{\frac{n}{q}} \int_{2r}^{\i}\|f\|_{L_p(B(x_0,t))}\frac{dt}{t^{\frac{n}{q}+1}}.
\end{align}
Thus
\begin{equation*}
\|I_{\Omega,\a} f\|_{L_q(B)} \lesssim \;  r^{\frac{n}{q}} \int_{2r}^{\i}\|f\|_{L_p(B(x_0,t))}\frac{dt}{t^{\frac{n}{q}+1}}.
\end{equation*}

When $1 < q < s$, by Fubini's theorem and the Minkowski inequality, we get
\begin{align} \label{ves2mPot}
\|I_{\Omega,\a}f_2\|_{L_q(B)} & \le
\Big( \int_{B} \Big| \int_{2r}^{\i} \int_{B(x_0,t)}|f(y)| |\Omega(x-y)| dy \frac{dt}{t^{n+1-\a}}\Big|^q \Big)^{\frac{1}{q}} \notag
\\
& \le \int_{2r}^{\i} \int_{B(x_0,t)}|f(y)| \, \|\Omega(\cdot-y)\|_{L_q(B)} dy \frac{dt}{t^{n+1-\a}}\notag
\\
& \le \, r^{\frac{n}{q}-\frac{n}{s}} \int_{2r}^{\i} \int_{B(x_0,t)}|f(y)| \, \|\Omega(\cdot-y)\|_{L_s(B)} dy \, \frac{dt}{t^{n+1-\a}} \notag
\\
& \lesssim r^{\frac{n}{q}} \int_{2r}^{\i}\|f\|_{L_1(B(x_0,t))} \frac{dt}{t^{n+1-\a}}
\\
& \lesssim r^{\frac{n}{q}} \int_{2r}^{\i}\|f\|_{L_p(B(x_0,t))} \frac{dt}{t^{\frac{n}{q}+1}}. \notag
\end{align}

Let $p=1<q<s \le \i$. From the weak $(1,q)$ boundedness of $I_{\Omega,\a}$ and \eqref{sal01Pot} it follows
that:
\begin{align} \label{gfr9Pot}
\|I_{\Omega,\a} f_1\|_{WL_q(B)} & \leq \|I_{\Omega,\a} f_1\|_{WL_q(\Rn)}\lesssim \|f_1\|_{L_1(\Rn)} \notag
\\
& = \|f\|_{L_1(2B)} \lesssim r^{\frac{n}{q}} \int_{2r}^{\i}\|f\|_{L_1(B(x_0,t))}\frac{dt}{t^{\frac{n}{q}+1}}.
\end{align}

Then from \eqref{ves2Pot} and \eqref{gfr9Pot} we get the inequality \eqref{eq3.5.WXM0}.
\end{proof}

\begin{thm}\label{3.4.Pot}
Suppose that $x_0 \in \Rn$, $\Omega \in L_s(S^{n-1})$, $ 1<s\leq\infty$, be a homogeneous of degree zero. Let $0 < \a < n$, $1 \le p < \frac{n}{\a}$,
$\frac{1}{q}=\frac{1}{p}-\frac{\a}{n}$, and $s' \le p$ or $q< s$. Let also, the pair $(\varphi_1,\varphi_2)$
satisfy the condition
\begin{equation}\label{eq3.6.VZPot}
\int_{r}^{\infty} \frac{\es_{t<\tau<\infty} \varphi_1(x_0,\tau) \tau^{\frac{n}{p}}}{t^{\frac{n}{q}+1}}dt \le
 C \,\varphi_2(x_0,r),
\end{equation}
where $C$ does not depend on $r$. Then the operators $M_{\Omega,\a}$ and $I_{\Omega,\a}$ are bounded from $LM_{p,\varphi_1}^{\{x_0\}}$ to $LM_{q,\varphi_2}^{\{x_0\}}$
for $p>1$ and from $LM_{1,\varphi_1}^{\{x_0\}}$ to $WLM_{q,\varphi_2}^{\{x_0\}}$ for $p=1$. Moreover, for $p>1$
\begin{equation*}
\|M_{\Omega,\a} f\|_{LM_{q,\varphi_2}^{\{x_0\}}} \lesssim
\|I_{\Omega,\a} f\|_{LM_{q,\varphi_2}^{\{x_0\}}} \lesssim \|f\|_{LM_{p,\varphi_1}^{\{x_0\}}}
\end{equation*}
and for $p=1$
\begin{equation*}
\|M_{\Omega,\a} f\|_{WLM_{q,\varphi_2}^{\{x_0\}}} \lesssim \|I_{\Omega,\a} f\|_{WLM_{q,\varphi_2}^{\{x_0\}}} \lesssim \|f\|_{LM_{1,\varphi_1}^{\{x_0\}}}.
\end{equation*}
\end{thm}
\begin{proof}
By Lemma \ref{lem3.3.Pot} and Theorem \ref{thm3.2.} with $v_2(r)=\varphi_2(x_0,r)^{-1}$,
$v_1(r)=\varphi_1(x_0,r)^{-1} r^{-\frac{n}{p}}$ and $w(r)=r^{-\frac{n}{q}}$ we have for $p>1$
\begin{equation*}
\begin{split}
\|I_{\Omega,\a} f\|_{LM_{q,\varphi_2}^{\{x_0\}}} & \lesssim \sup_{r>0}\varphi_2(x_0,r)^{-1}
\int_r^{\i}\|f\|_{L_p(B(x_0,t))}\frac{dt}{t^{\frac{n}{q}+1}}
\\
& \lesssim \sup_{r>0}\varphi_1(x_0,r)^{-1} \, r^{-\frac{n}{p}} \,\|f\|_{L_p(B(x_0,r))} = \|f\|_{LM_{p,\varphi_1}^{\{x_0\}}}
\end{split}
\end{equation*}
and for $p=1$
\begin{equation*}
\begin{split}
\|I_{\Omega,\a} f\|_{WLM_{q,\varphi_2}^{\{x_0\}}} & \lesssim \sup_{r>0}\varphi_2(x_0,r)^{-1}
\int_r^{\i}\|f\|_{L_1(B(x_0,t))}\frac{dt}{t^{\frac{n}{q}+1}}
\\
& \lesssim \sup_{r>0}\varphi_1(x_0,r)^{-1} \, r^{-n} \, \|f\|_{L_p(B(x_0,r))} = \|f\|_{LM_{1,\varphi_1}^{\{x_0\}}}.
\end{split}
\end{equation*}

\end{proof}
\begin{cor}\label{3.4.PotG}
Suppose that $\Omega \in L_s(S^{n-1})$, $ 1<s\leq\infty$, be a homogeneous of degree zero. Let $0 < \a < n$, $1 \le p < \frac{n}{\a}$,
$\frac{1}{q}=\frac{1}{p}-\frac{\a}{n}$, and $s' \le p$ or $q< s$. Let also, the pair $(\varphi_1,\varphi_2)$ satisfy the condition
\begin{equation*}\label{eq3.6.VZPotG}
\int_{r}^{\infty} \frac{\es_{t<\tau<\infty} \varphi_1(x,\tau) \tau^{\frac{n}{p}}}{t^{\frac{n}{q}+1}}dt \le C \,\varphi_2(x,r),
\end{equation*}
where $C$ does not depend on $x$ and $r$. Then the operators $M_{\Omega,\a}$ and $I_{\Omega,\a}$ are bounded from $M_{p,\varphi_1}$ to $M_{q,\varphi_2}$
for $p>1$ and from $M_{1,\varphi_1}$ to $WM_{q,\varphi_2}$ for $p=1$. Moreover, for $p>1$
\begin{equation*}
\|M_{\Omega,\a} f\|_{M_{q,\varphi_2}} \lesssim \|I_{\Omega,\a} f\|_{M_{q,\varphi_2}} \lesssim \|f\|_{M_{p,\varphi_1}},
\end{equation*}
and for $p=1$
\begin{equation*}
\|M_{\Omega,\a} f\|_{WM_{q,\varphi_2}} \lesssim \|I_{\Omega,\a} f\|_{WM_{q,\varphi_2}} \lesssim \|f\|_{M_{1,\varphi_1}}.
\end{equation*}
\end{cor}

\begin{cor} \label{gar1Pot}
Let $1 \le p < \infty$, $0<\a<\frac{n}{p}$, $\frac{1}{q}=\frac{1}{p}-\frac{\a}{n}$
 and $(\varphi_1,\varphi_2)$ satisfy condition \eqref{eq3.6.VZPot}.
Then the operators $M_{\a}$  and $I_{\a}$ are bounded from $LM_{p,\varphi_1}^{\{x_0\}}$ to
$LM_{q,\varphi_2}^{\{x_0\}}$ for $p>1$ and from $M_{1,\varphi_1}^{\{x_0\}}$ to $WLM_{q,\varphi_2}^{\{x_0\}}$ for $p=1$.
\end{cor}
\begin{rem}
Note that, in the case $s=\i$ Corollary \ref{3.4.PotG} was proved in \cite{GULAKShIEOT2012}. The condition \eqref{eq3.6.VZPot} in Theorem \ref{3.4.Pot} is weaker
than condition \eqref{GulZPot} in Theorem \ref{nakaiVagifPot0} (see \cite{GULAKShIEOT2012}). 
\end{rem}

\

\section{ Commutators of fractional integral operator with rough kernels in the spaces $LM_{p,\varphi}^{\{x_0\}}$}

Let $T$ be a linear operator, for a function $b$, we define the commutator $[b, T]$ by
$$
[b, T] f(x) = b(x) \, Tf(x) - T(b f)(x)
$$
for any suitable function $f$. If $\widetilde{T}$ be a Calder\'{o}n-Zygmund singular integral operator, a well known result of Coifman, Rochberg and Weiss \cite{CRW} states that the commutator $[b, \widetilde{T}] f = b \, \widetilde{T}f-\widetilde{T}(b f)$ is bounded on $L_p(\Rn)$, $1 < p < \infty$, if and only if $b \in BMO(\Rn)$.
The commutator of Calder\'{o}n-Zygmund operators plays an important role in studying the regularity of solutions of elliptic partial differential equations of second order
(see, for example, \cite{ChFraL1, ChFraL2, FazRag2}). In \cite{Chanillo}, Chanillo proved that the commutator $[b,I_{\a}] f = b \, I_{\a}f-I_{\a}(b f)$
is bounded from $L_p(\Rn)$ to $L_q(\Rn)$, $(1 < p < q < \infty$, $\frac{1}{q}=\frac{1}{p}-\frac{\a}{n})$
if and only if $b \in BMO(\Rn)$.

%
%
%
%
%

The definition of local Campanato space as follows.

\begin{defn}
Let $1\le q<\i$ and $0\le\lambda<\frac{1}{n}$. A function $f\in L_q^{\rm loc}(\Rn)$ is said to belong to the $CBMO_{q,\lambda}^{\{x_0\}}(\Rn)$ (central Campanato space), if
\begin{equation*}
\|f\|_{CBMO_{q,\lambda}^{\{x_0\}}}=\sup_{r>0} \Big(\frac{1}{|B(x_0,r)|^{1+\lambda q}} \int_{B(x_0,r)}|f(y)-f_{B(x_0,r)}|^q dy \Big)^{1/q}<\infty,
\end{equation*}
where
$$
f_{B(x_0,r)}=\frac{1}{|B(x_0,r)|} \int_{B(x_0,r)} f(y)dy.
$$
Define
$$
CBMO_{q,\lambda}^{\{x_0\}}(\Rn)=\{f \in L_q^{\rm loc}(\Rn) ~ : ~ \| f \|_{CBMO_{q,\lambda}^{\{x_0\}}} < \i  \}.
$$
\end{defn}
In \cite{LuYang1}, Lu and Yang introduced the central BMO space $CBMO_{q}(\Rn)=CBMO_{q,0}^{\{0\}}(\Rn)$.
Note that, $BMO(\Rn) \subset CBMO_{q}^{\{x_0\}}(\Rn)$, $1\le q<\i$.
The space $CBMO_q^{\{x_0\}}(\Rn)$ can be regarded as a local version of $BMO(\Rn)$, the space of bounded
mean oscillation, at the origin. But, they have quite different properties. The classical
John-Nirenberg inequality shows that functions in $BMO(\Rn)$ are locally exponentially integrable.
This implies that, for any $1\le q<\i$, the functions in $BMO(\Rn)$ can be described by means of the condition:
\begin{equation*}
\sup_{r>0} \Big( \frac{1}{|B|} \int_{B}|f(y)-f_{B}|^qdy \Big)^{1/q}<\infty,
\end{equation*}
where $B$ denotes an arbitrary ball in $\Rn$. However, the space $CBMO_q^{\{x_0\}}(\Rn)$ depends on $q$.
If $q_1 < q_2$, then $CBMO_{q_2}^{\{x_0\}}(\Rn) \subsetneqq CBMO_{q_1}^{\{x_0\}}(\Rn)$.
Therefore, there is no analogy of the famous John-Nirenberg inequality of $BMO(\Rn)$ for the space $CBMO_q^{\{x_0\}}(\Rn)$.
One can imagine that the behavior of $CBMO_q^{\{x_0\}}(\Rn)$ may be quite different from that of $BMO(\Rn)$.

\begin{lem} 
Let $b$ be a function in $CBMO_{q,\lambda}^{\{x_0\}}(\Rn)$, $1 \le q < \infty$, $0\le\lambda<\frac{1}{n}$ and $r_1, r_2 > 0$. Then
$$
\left( \frac{1}{|B(x_0,r_1)|^{1+\lambda q}} \int_{B(x_0,r_1)} |b(y)-b_{B(x_0,r_2)}|^q dy\right)^{\frac{1}{q}}
\le C \left( 1+ \Big|\ln\frac{r_1}{r_2} \Big| \right) \| b \|_{{CBMO_{q,\lambda}^{\{x_0\}}}},
$$
where $C>0$ is independent of $b$, $r_1$ and $r_2$.
\end{lem}

In \cite{DingYZ} the following statement was proved for the commutators of fractional integral operators with
rough kernels, containing the result in \cite{Miz, Nakai}.
\begin{thm}  \label{nakaiFFPot}
Suppose that $x_0 \in \Rn$, $\Omega \in L_s(S^{n-1})$, $ 1<s\leq\infty$, be a homogeneous of degree zero and $b \in BMO(\Rn)$.
Let $0<\a<n$, $1 \le s' < p < \frac{n}{p}$, $\frac{1}{q}=\frac{1}{p}-\frac{\a}{n}$, $\varphi(x,r)$ which satisfies the conditions \eqref{nakcond} and \eqref{MizNPot}.
 Then the operator $[b,I_{\Omega,\a}]$ is bounded from $M_{p,\varphi}$ to $M_{q,\varphi}$.
\end{thm}

\begin{lem}\label{lem5.1.Pot}
Suppose that $x_0 \in \Rn$, $\Omega \in L_s(S^{n-1})$, $ 1<s\leq\infty$, be a homogeneous of degree zero.
Let $0<\a<n$, $1<p<\frac{n}{\a}$, $b \in CBMO_{p_2,\lambda}^{\{x_0\}}(\Rn)$, $0\le\lambda<\frac{1}{n}$,
$\frac{1}{p}=\frac{1}{p_1}+\frac{1}{p_2}$, $\frac{1}{q}=\frac{1}{p}-\frac{\a}{n}$, $\frac{1}{q_1}=\frac{1}{p_1}-\frac{\a}{n}$.

Then, for $s' \le p$ or $q_1 < s$ the inequality
\begin{equation*}\label{eq5.1.}
\|[b, I_{\Omega,\a}]f\|_{L_q(B(x_0,r))} \lesssim \|b\|_{CBMO_{p_2,\lambda}^{\{x_0\}}} \, r^{\frac{n}{q}}
\int_{2r}^{\i} \Big(1+\ln \frac{t}{r}\Big)
t^{n\lambda-\frac{n}{q_1}-1} \|f\|_{L_{p_1}(B(x_0,t))} dt
\end{equation*}
holds for any ball $B(x_0,r)$ and for all $f \in L_{p_1}^{loc}(\Rn)$.

\end{lem}
\begin{proof}
Let $1 < p < \infty$, $0<\a<\frac{n}{p}$, $\frac{1}{p}=\frac{1}{p_1}+\frac{1}{p_2}$, $\frac{1}{q}=\frac{1}{p}-\frac{\a}{n}$, and $\frac{1}{q_1}=\frac{1}{p_1}-\frac{\a}{n}$.
As in the proof of Lemma \ref{lem3.3.Pot}, we represent function $f$ in form
\eqref{repr} and have
\begin{align*}
[b, I_{\Omega,\a}]f(x) & = \big(b(x)-b_{B}\big) I_{\Omega,\a}f_{1}(x) - I_{\Omega,\a} \Big(\big(b(\cdot)-b_{B}\big)f_{1}\Big)(x)
\\
& + \big(b(x)-b_{B}\big) I_{\Omega,\a}f_{2}(x) - I_{\Omega,\a} \Big(\big(b(\cdot)-b_{B}\big)f_{2}\Big)(x)
\\
& \equiv J_1 + J_2 + J_3 + J_4.
\end{align*}

Hence we get
$$
\|[b, I_{\Omega,\a}]f\|_{L_q(B)}\leq \|J_1\|_{L_q(B)}+\|J_2\|_{L_q(B)}+\|J_3\|_{L_q(B)}+\|J_4\|_{L_q(B)}.
$$
From the boundedness of $[b, I_{\Omega,\a}]$ from $L_{p_1}(\Rn)$ to $L_{q_1}(\Rn)$ it follows that:
\begin{align*}
\|J_1\|_{L_q(B)} & \leq \| \big(b(\cdot)-b_{B}\big) [b, I_{\Omega,\a}]f_{1}(\cdot) \|_{L_{q}(\Rn)}
\\
& \le \| \big(b(\cdot)-b_{B}\big) \|_{L_{p_2}(\Rn)} [b, I_{\Omega,\a}]f_{1}(\cdot) \|_{L_{q_1}(\Rn)}
\\
& \leq C \|b\|_{CBMO_{p_2,\lambda}^{\{x_0\}}} \, r^{\frac{n}{p_2}+n\lambda} \,  \| f_1\|_{L_{p_1}(\Rn)}
\\
& = C \|b\|_{CBMO_{p_2,\lambda}^{\{x_0\}}} \, r^{\frac{n}{p_2}+\frac{n}{q_1}+n\lambda} \,  \| f\|_{L_{p_1}(2B)}
\int_{2r}^{\infty} t^{-1-\frac{n}{q_1}} dt
\\
& \lesssim \|b\|_{CBMO_{p_2,\lambda}^{\{x_0\}}} \, r^{\frac{n}{q}+n\lambda} \,
\int_{2r}^{\infty} \Big(1+\ln \frac{t}{r}\Big) \| f \|_{L_{p_1}(B(x_0,t))} t^{-1-\frac{n}{q_1}} dt.
\end{align*}

For $J_2$ we have
\begin{align*}
\| J_2\|_{L_{q}(B)} & \leq \| [b, I_{\Omega,\a}]\big(b(\cdot)-b_{B}\big) f_{1} \|_{L_{q}(\Rn)}
\\
& \lesssim \| (b(\cdot)-b_{B}) f_{1}| \|_{L_{p}(\Rn)}
\\
& \lesssim \| b(\cdot)-b_{B}\|_{L_{p_2}(\Rn)} \| f_{1} \|_{L_{p_1}(\Rn)}
\\
& \lesssim \|b\|_{CBMO_{p_2,\lambda}^{\{x_0\}}} \, r^{\frac{n}{p_2}+\frac{n}{q_1}+n\lambda} \,  \| f\|_{L_{p_1}(2B)}
\int_{2r}^{\infty} t^{-1-\frac{n}{q_1}} dt
\\
& \lesssim \|b\|_{CBMO_{p_2,\lambda}^{\{x_0\}}} \, r^{\frac{n}{p}+n\lambda} \,
\int_{2r}^{\infty} \Big(1+\ln \frac{t}{r}\Big) \| f \|_{L_{p_1}(B(x_0,t))} t^{-1-\frac{n}{q_1}} dt.
\end{align*}

For $J_3$, it is known that $x\in B$, $y\in {\dual (2B)}$,
which implies $\frac{1}{2}|x_0-y| \leq |x-y| \le \frac{3}{2}|x_0-y|$.

When $s' \le p$, by Fubini's theorem and applying H\"{o}lder inequality we have
\begin{align*}
|I_{\Omega,\a}f_2(x)| & \leq c_0 \int_{\dual (2B)} |\Omega(x-y)| \frac{|f(y)|}{|x_0-y|^{n-\a}}dy
\\
& \approx \int_{2r}^{\infty}\int_{2r<|x_0-y|<t} |\Omega(x-y)| |f(y)|dy \, t^{-1-n-\a} dt \notag
\\
&\lesssim \int_{2r}^{\infty}\int_{B(x_0,t)} |\Omega(x-y)| |f(y)|dy \, t^{-1-n-\a} dt  \notag
\\
& \lesssim \int_{2r}^{\infty} \| f \|_{L_{p_1}(B(x_0,t))} \, \|\Omega(x-\cdot)\|_{L_{s}(B(x_0,t))}
\, |B(x_0,t)|^{1-\frac{1}{p_1}-\frac{1}{s}}  \, t^{-1-\frac{n}{p_1}-\a} dt
\\
& \lesssim \int_{2r}^{\infty} \| f \|_{L_{p_1}(B(x_0,t))} \, t^{-1-\frac{n}{q_1}} dt.
\end{align*}

Hence, we get
\begin{align*}
\| J_3\|_{L_{q}(B)} & = \| \big(b(\cdot)-b_{B}\big) I_{\Omega,\a}f_{2}(\cdot) \|_{L_{q}(\Rn)}
\\
& \le \| \big(b(\cdot)-b_{B}\big) \|_{L_{q}(\Rn)} \int_{2r}^{\infty}
\| f \|_{L_{p_1}(B(x_0,t))} \, t^{-1-\frac{n}{q_1}} dt
\\
& \le \| \big(b(\cdot)-b_{B}\big) \|_{L_{p_2}(\Rn)} \, r^{\frac{n}{q_1}} \, \int_{2r}^{\infty}
\| f \|_{L_{p_1}(B(x_0,t))} \, t^{-1-\frac{n}{q_1}} dt
\\
& \lesssim \|b\|_{CBMO_{p_2,\lambda}^{\{x_0\}}} \, r^{\frac{n}{q}+n\lambda} \,
\int_{2r}^{\infty} \Big(1+\ln \frac{t}{r}\Big) \| f \|_{L_{p_1}(B(x_0,t))} t^{-1-\frac{n}{q_1}} dt.
\end{align*}

When $q_1 < s$, by Fubini's theorem and the Minkowski inequality, we get
\begin{align} \label{ves2m}
\|J_3\|_{L_q(B)} & \le
\Big( \int_{B} \big| \int_{2r}^{\i} \int_{B(x_0,t)}|f(y)| |b(x)-b_{B}| |\Omega(x-y)| dy
\frac{dt}{t^{n-\a+1}}\big|^q \Big)^{\frac{1}{q}} \notag
\\
& \le \int_{2r}^{\i} \int_{B(x_0,t)}|f(y)| \, \|(b(\cdot)-b_{B}) \Omega(\cdot-y)\|_{L_q(B)} dy \frac{dt}{t^{n-\a+1}} \notag
\\
& \le \int_{2r}^{\i} \int_{B(x_0,t)}|f(y)| \, \|b(\cdot)-b_{B}\|_{L_{p_2}(B)}
\, \|\Omega(\cdot-y)\|_{L_{q_1}(B)} dy  \, \frac{dt}{t^{n-\a+1}} \notag
\\
& \lesssim  \|b\|_{CBMO_{p_2,\lambda}^{\{x_0\}}} \, r^{\frac{n}{p_2}+n\lambda} \, |B|^{\frac{1}{q_1}-\frac{1}{s}} \, \int_{2r}^{\i} \int_{B(x_0,t)}|f(y)| \, \|\Omega(\cdot-y)\|_{L_{s}(B)} dy \, \frac{dt}{t^{n-\a+1}} \notag
\\
& \lesssim \|b\|_{CBMO_{p_2,\lambda}^{\{x_0\}}} \, r^{\frac{n}{q}+n\lambda} \int_{2r}^{\i}\|f\|_{L_1(B(x_0,t))} \frac{dt}{t^{n-\a+1}}
\\
& \lesssim \|b\|_{CBMO_{p_2,\lambda}^{\{x_0\}}} \, r^{\frac{n}{q}+n\lambda} \int_{2r}^{\i} \Big(1+\ln \frac{t}{r}\Big) \|f\|_{L_{p_1}(B(x_0,t))} \frac{dt}{t^{\frac{n}{q_1}+1}}. \notag
\end{align}

For $x\in B$ by Fubini's theorem and applying H\"{o}lder inequality we have
\begin{align*}
& |I_{\Omega,\a} \Big(\big(b(\cdot)-b_{B}\big)f_{2}\Big)(x)|
\lesssim \int_{\dual (2B)} |b(y)-b_{B}| \, |\Omega(x-y)| \, \frac{|f(y)|}{|x-y|^{n-\a}}dy
\\
& \lesssim \int_{\dual (2B)} |b(y)-b_{B}| \, |\Omega(x-y)| \, \frac{|f(y)|}{|x_0-y|^{n-\a}}dy
\\
& \approx \int_{2r}^{\infty}\int_{2r<|x_0-y|<t}|b(y)-b_{B}| \, |\Omega(x-y)| \, |f(y)|dy \, t^{\a-n-1} dt \notag
\\
& \lesssim \int_{2r}^{\i}\int_{B(x_0,t)} |b(y)-b_{B(x_0,t)}||\Omega(x-y)| \, |f(y)|dy\frac{dt}{t^{n-\a+1}}
\\
&~~~~~~~~~ + \int_{2r}^{\i}|b_{B(x_0,r)}-b_{B(x_0,t)}|
\int_{B(x_0,t)} |\Omega(x-y)| \, |f(y)|dy\frac{dt}{t^{n-\a+1}}
\\
&\lesssim  \int_{2r}^{\i} \, \|(b(\cdot)-b_{B(x_0,t)}) f \|_{L_p(B(x_0,t))}
\, \|\Omega(\cdot-y)\|_{L_{s}(B(x_0,t))} \, |B(x_0,t)|^{1-\frac{1}{p}-\frac{1}{s}}  \, \frac{dt}{t^{n-\a+1}}
\end{align*}
\begin{align*}
& + \int_{2r}^{\i}|b_{B(x_0,r)}-b_{B(x_0,t)}| \|f\|_{L_{p_1}(B(x_0,t))}
\, \|\Omega(\cdot-y)\|_{L_{s}(B(x_0,t))} \, |B(x_0,t)|^{1-\frac{1}{p_1}-\frac{1}{s}}  \, t^{\a-n-1} dt
\\
&\lesssim  \int_{2r}^{\i} \, \|b(\cdot)-b_{B(x_0,t)} \|_{L_{p_2}(B(x_0,t))} \|f\|_{L_{p_1}(B(x_0,t))}
t^{-1-\frac{n}{q_1}}  \, dt
\\
& + \|b\|_{CBMO_{p_2,\lambda}^{\{x_0\}}} \, \int_{2r}^{\i} \Big(1+\ln \frac{t}{r}\Big) \|f\|_{L_{p_1}(B(x_0,t))}
 \, t^{n\lambda-1-\frac{n}{q_1}} dt
\\
& \lesssim \|b\|_{CBMO_{p_2,\lambda}^{\{x_0\}}} \, \int_{2r}^{\i}\Big(1+\ln \frac{t}{r}\Big)
\|f\|_{L_{p_1}(B(x_0,t))} ~ t^{n\lambda-1-\frac{n}{q_1}} dt.
\end{align*}
Then for $J_4$ we have
\begin{align*}
\| J_4\|_{L_{q}(B)} & \leq \| I_{\Omega,\a}\big(b(\cdot)-b_{B}\big) f_2 \|_{L_{q}(\Rn)}
\\
& \lesssim \|b\|_{CBMO_{p_2,\lambda}^{\{x_0\}}} \, r^{\frac{n}{q}} \,
\int_{2r}^{\infty} \Big(1+\ln \frac{t}{r}\Big) \| f \|_{L_{p_1}(B(x_0,t))} t^{n\lambda-1-\frac{n}{q_1}} dt.
\end{align*}

When $q_1 < s$, by Fubini's theorem and the Minkowski inequality, we get

\begin{align} \label{ves2m}
\|I_{\Omega,\a}f_2\|_{L_q(B)} & \le
\Big( \int_{B} \big| \int_{2r}^{\i} \int_{B(x_0,t)}|f(y)| |\Omega(x-y)| dy \frac{dt}{t^{n-a+1}}\big|^q \Big)^{\frac{1}{q}} \notag
\\
& \le \int_{2r}^{\i} \int_{B(x_0,t)}|f(y)| \, \|\Omega(\cdot-y)\|_{L_q(B)} dy \frac{dt}{t^{n-a+1}}\notag
\\
& \le \, |B|^{\frac{1}{q}-\frac{1}{s}} \, \int_{2r}^{\i} \int_{B(x_0,t)}|f(y)| \, \|\Omega(\cdot-y)\|_{L_s(B)} dy
 \, \frac{dt}{t^{n-a+1}} \notag
\\
& \lesssim r^{\frac{n}{q}} \int_{2r}^{\i}\|f\|_{L_1(B(x_0,t))} \frac{dt}{t^{n-a+1}}
\\
& \lesssim r^{\frac{n}{q}} \int_{2r}^{\i}\|f\|_{L_{p_1}(B(x_0,t))} \frac{dt}{t^{\frac{n}{q_1}+1}}. \notag
\end{align}

Now combined by all the above estimates, we end the proof of this Lemma \ref{lem5.1.Pot}.

\end{proof}

The following theorem is true.
\begin{thm} \label{theor3.3FPot}
Suppose that $x_0 \in \Rn$, $\Omega \in L_s(S^{n-1})$ with $ 1<s\leq\infty$, be a homogeneous of degree zero.
Let $0<\a<n$, $1<p<\frac{n}{\a}$, $b \in CBMO_{p_2,\lambda}^{\{x_0\}}(\Rn)$, $0\le\lambda<\frac{1}{n}$,
$\frac{1}{p}=\frac{1}{p_1}+\frac{1}{p_2}$, $\frac{1}{q}=\frac{1}{p}-\frac{\a}{n}$, $\frac{1}{q_1}=\frac{1}{p_1}-\frac{\a}{n}$.
Let also, for $s' \le p$ or $q_1 < s$ the pair $(\varphi_1,\varphi_2)$
satisfy the condition
\begin{equation}\label{eq3.6.VZPotcomm}
\int_{r}^{\infty} \Big(1+\ln \frac{t}{r}\Big) \,
\frac{\es_{t<\tau<\infty} \varphi_1(x_0,\tau) \tau^{\frac{n}{p}}}{t^{\frac{n}{q}-n\lambda+1}}dt \le
 C \,\varphi_2(x_0,r),
\end{equation}
where $C$ does not depend on $r$. Then, the operators $M_{\Omega,b,\a}$ and $[b, I_{\Omega,\a}]$ are bounded from $LM_{p,\varphi_1}^{\{x_0\}}$ to $LM_{q,\varphi_2}^{\{x_0\}}$. Moreover
\begin{equation*}
\|M_{\Omega,b,\a}f\|_{LM_{q,\varphi_2}^{\{x_0\}}} \lesssim \|[b, I_{\Omega,\a}]f\|_{LM_{q,\varphi_2}^{\{x_0\}}} \lesssim \|b\|_{CBMO_{p_2,\lambda}^{\{x_0\}}} \, \|f\|_{LM_{p,\varphi_1}^{\{x_0\}}}.
\end{equation*}
\end{thm}
\begin{proof}
The statement of Theorem \ref{theor3.3FPot} follows by Lemma \ref{lem5.1.Pot} and Theorem \ref{thm3.2.} in the same manner as in the proof of Theorem \ref{3.4.Pot}.
\end{proof}
%

For the sublinear commutator of the fractional maximal operator $M_{b,\a}$
and for the linear commutator of the Riesz potential $[b,I_{\a}]$ from Theorem \ref{theor3.3FPot} we get the following new results.
\begin{cor} \label{gar2CPot}
Let $0<\a<n$, $1<p<\frac{n}{\a}$, $b \in CBMO_{p_2,\lambda}^{\{x_0\}}(\Rn)$, $0\le\lambda<\frac{1}{n}$,
$\frac{1}{p}=\frac{1}{p_1}+\frac{1}{p_2}$, $\frac{1}{q}=\frac{1}{p}-\frac{\a}{n}$, $\frac{1}{q_1}=\frac{1}{p_1}-\frac{\a}{n}$, and $(\varphi_1,\varphi_2)$ satisfies the condition \eqref{eq3.6.VZPotcomm}. Then, the operators $M_{b,\a}$ and $[b,I_{\a}]$ are bounded from $LM_{p_1,\varphi_1}^{\{x_0\}}$ to $LM_{q,\varphi_2}^{\{x_0\}}$.
\end{cor}

\section{Some applications}

In this section, we shall apply Theorems \ref{3.4.Pot} and \ref{theor3.3FPot} to several
particular operators such as the Marcinkiewicz operator and
fractional powers of the some analytic semigroups.

\subsection{ Marcinkiewicz operator}

Let $S^{n-1}=\{x\in\Rn:|x|=1\}$ be the unit sphere in $\Rn$ equipped with the Lebesgue measure $d\sigma$.
Suppose that $x_0 \in \Rn$, $\Omega \in L_s(S^{n-1})$, $ 1<s\leq\infty$, be a homogeneous of degree zero and satisfy the cancellation condition.

In 1958, Stein \cite{Stein58} defined the Marcinkiewicz integral of higher dimension $\mu_\Omega$ as
\begin{equation*}
\mu_\Omega(f)(x)=\left(\int_0^\infty|F_{\Omega,t}(f)(x)|^2\frac{dt}{t^3}\right)^{1/2},
\end{equation*}
where
$$
F_{\Omega,t}(f)(x)=\int_{|x-y|\leq t}\frac{\Omega(x-y)}{|x-y|^{n-1}}f(y)dy.
$$

Since Stein's work in 1958, the continuity of Marcinkiewicz integral has been extensively studied
as a research topic and also provides useful tools in harmonic analysis \cite{LuDingY, St, Stein93, Torch}.

The Marcinkiewicz operator is defined by (see \cite{TorWang})
\begin{equation*}
\mu_{\Omega,\a}(f)(x)=\left(\int_0^\infty|F_{\Omega,\a,t}(f)(x)|^2\frac{dt}{t^3}\right)^{1/2},
\end{equation*}
where 
$$
F_{\Omega,\a,t}(f)(x)=\int_{|x-y|\leq t}\frac{\Omega(x-y)}{|x-y|^{n-1-\a}}f(y)dy.
$$
Note that $\mu_{\Omega}f=\mu_{\Omega,0}f$.

Let $H$ be the space
$H=\{h:\|h\|=(\int_0^\infty|h(t)|^2dt/t^3)^{1/2}<\i\}$. Then, it is
clear that $\mu_{\Omega,\a}(f)(x)=\|F_{\Omega,\a,t}(x)\|$.

By Minkowski inequality and the conditions on $\Omega$, we get
\begin{equation*}
\mu_{\Omega,\a}(f)(x)\leq\int_{\Rn}\frac{|\Omega(x-y)|}{|x-y|^{n-1-\a}}|f(y)|
\left(\int_{|x-y|}^\infty\frac{dt}{t^3}\right)^{1/2} dy
\leq C I_{\Omega,\a}(f)(x).
\end{equation*}
It is known that $\mu_{\Omega,\a}$ is bounded from $L_p(\Rn)$ to $L_q(\Rn)$ for $p>1$,
and bounded from $L_1(\Rn)$ to $WL_q(\Rn)$ for $p=1$
(see \cite{TorWang}), then from Theorems \ref{3.4.Pot} and \ref{theor3.3FPot} we get
\begin{cor}\label{GARM1}
Suppose that $x_0 \in \Rn$, $\Omega \in L_s(S^{n-1})$, $ 1<s\leq\infty$, be a homogeneous of degree zero and satisfy the cancellation condition. Let $0 < \a < n$, $1 \le p < \frac{n}{\a}$, $\frac{1}{q}=\frac{1}{p}-\frac{\a}{n}$ and for $s' \le p$ or $q_1 < s$ the pair $(\varphi_1,\varphi_2)$ satisfy the condition \eqref{eq3.6.VZPot}. Then $\mu_{\Omega,\a}$ is bounded from $LM_{p,\varphi_1}^{\{x_0\}}$ to $LM_{q,\varphi_2}^{\{x_0\}}$ for $p>1$ and from $M_{1,\varphi_1}^{\{x_0\}}$ to $WLM_{q,\varphi_2}^{\{x_0\}}$ for $p=1$.
\end{cor}
\begin{cor}\label{GARM2}
Suppose that $x_0 \in \Rn$, $\Omega \in L_s(S^{n-1})$, $ 1<s\leq\infty$, be a homogeneous of degree zero and satisfy the cancellation condition. Let $0<\a<n$, $1<p<\frac{n}{\a}$, 
$b \in CBMO_{p_2,\lambda}^{\{x_0\}}(\Rn)$, $0\le\lambda<\frac{1}{n}$, $\frac{1}{p}=\frac{1}{p_1}+\frac{1}{p_2}$, $\frac{1}{q}=\frac{1}{p}-\frac{\a}{n}$, $\frac{1}{q_1}=\frac{1}{p_1}-\frac{\a}{n}$ and for $s' \le p$  
or $q_1 < s$ the pair $(\varphi_1,\varphi_2)$ satisfy the condition \eqref{eq3.6.VZPot}. Then $[a,\mu_{\Omega,\a}]$ is bounded from $LM_{p,\varphi_1}^{\{x_0\}}$ to $LM_{q,\varphi_2}^{\{x_0\}}$.
\end{cor}

\

\subsection{Fractional powers of the some analytic semigroups}

The theorems of the previous sections can be applied to various operators which are estimated from above by Riesz potentials. We give some examples.

Suppose that $L$ is a linear operator on $L_2$ which generates an analytic semigroup $e^{-tL}$ with the kernel $p_t(x,y)$ satisfying a Gaussian upper bound, that is,
\begin{equation}\label{kern0}
|p_t(x,y)|\leq {c_1 \over{ t^{{n}/{2}} }} e^{-c_2{{|x-y|^2}\over t}}
\end{equation}
for  $x,y\in {\mathbb R}^n$ and all $t>0$, where $c_1,\, c_2 > 0$ are independent of $x$, $y$ and $t$.

For $0<\alpha<n,$ the fractional powers $L^{-\alpha/2}$ of the operator $L$ are defined by
\begin{equation*}\label{dy1}
L^{-\alpha/2}f(x)={1\over {\Gamma(\alpha/2)}}\int_0^{\infty} e^{-tL}f (x)\frac{dt}{t^{-\alpha/2+1}}.
\end{equation*}

Note that if $L=-\triangle$ is the Laplacian on  ${\mathbb  R}^n$,
then $L^{-\alpha/2}$ is the Riesz potential $I_{\alpha}$. See, for
example, Chapter 5 in \cite{St}.

\begin{thm}\label{DY}
Let condition \eqref{kern0} be satisfied. Moreover, let $1 \le p < \infty$,
$0<\a<\frac{n}{p}$, $\frac{1}{q}=\frac{1}{p}-\frac{\a}{n}$,
$(\varphi_1,\varphi_2)$ satisfy condition \eqref{eq3.6.VZPot}.
Then $L^{-\alpha/2}$ is bounded from $LM_{p,\varphi_1}^{\{x_0\}}$ to
$LM_{q,\varphi_2}^{\{x_0\}}$ for $p>1$ and from $M_{1,\varphi_1}^{\{x_0\}}$ to $WLM_{q,\varphi_2}^{\{x_0\}}$ for $p=1$.
\end{thm}
\begin{proof}
Since the  semigroup $e^{-tL}$ has the kernel $p_t(x,y)$ which
satisfies condition \eqref{kern0}, it follows that
$$
|L^{-\alpha/2}f(x)| \lesssim I_{\alpha}(|f|)(x)
$$
(see \cite{DY}). Hence by the aforementioned theorems we have
$$
\|L^{-\alpha/2}f\|_{M_{q,\varphi_2}^{\{x_0\}}} \lesssim
\|I_{\alpha}(|f|)\|_{M_{q,\varphi_2}^{\{x_0\}}} \lesssim
\|f\|_{M_{p,\varphi_1}^{\{x_0\}}}.
$$
\end{proof}

Let $b$ be a locally integrable function on $\Rn$, the commutator of $b$ and $L^{-\alpha/2}$
is defined as follows
$$
[b,L^{-\alpha/2}]f(x)=b(x) L^{-\alpha/2}f(x) - L^{-\alpha/2}(bf)(x).
$$

In \cite{DY} extended the result of \cite{Chanillo} from $(-\Delta)$ to the more general operator
$L$ defined above. More precisely, they showed that when $b \in BMO(\Rn)$, then the commutator
operator $[b,L^{-\alpha/2}]$ is bounded from $L_p(\Rn)$ to $L_q(\Rn)$ for $1<p<q<\infty$ and $\frac{1}{q}=\frac{1}{p}-\frac{\a}{n}$. Then from Theorem \ref{theor3.3FPot} we get

\begin{thm}\label{DYCom}
Let condition \eqref{kern0} be satisfied. Moreover, let
$0<\a<n$, $1<p<\frac{n}{\a}$, $b \in CBMO_{p_2,\lambda}^{\{x_0\}}(\Rn)$, $0\le\lambda<\frac{1}{n}$, 
$\frac{1}{p}=\frac{1}{p_1}+\frac{1}{p_2}$, $\frac{1}{q}=\frac{1}{p}-\frac{\a}{n}$, and $\frac{1}{q_1}=\frac{1}{p_1}-\frac{\a}{n}$, and
$(\varphi_1,\varphi_2)$ satisfies the condition \eqref{eq3.6.VZPotcomm}.
Then $[b,L^{-\alpha/2}]$ is bounded from $LM_{p,\varphi_1}^{\{x_0\}}$ to $LM_{q,\varphi_2}^{\{x_0\}}$.
\end{thm}

Property \eqref{kern0} is satisfied for large classes of
differential operators (see, for example \cite{BurGul2}).
In \cite{BurGul2} also other examples of operators which are estimates
from above by Riesz potentials are given. In these cases Theorem \ref{3.4.Pot} and \ref{theor3.3FPot}
are also applicable for proving boundedness of those operators and commutators from $LM_{p,\varphi_1}^{\{x_0\}}$ to
$LM_{q,\varphi_2}^{\{x_0\}}$.

\



\begin{thebibliography}{99}
\bibitem{Adams} D.R. Adams, \textit{A note on Riesz potentials},
Duke Math. {\bf 42} (1975), 765-778.

\bibitem{AkbGulMus} Ali Akbulut, V.S. Guliyev and R. Mustafayev,
\textit{ Boundedness of the maximal operator and
singular integral operator in generalized Morrey spaces},
Mathematica Bohemica, 137 (1) 2012, 27-43.

\bibitem{AlvLanLakey} J. Alvarez, M. Guzman-Partida, J. Lakey,
\textit{Spaces of bounded $\lambda$-central mean oscillation, Morrey spaces,
and $\lambda$-central Carleson measures}, Collect. Math., 51 (2000), 1-47.

\bibitem{Beurl} A. Beurling,
\textit{Construction and analysis of some convolution algebras},
Ann. Inst. Fourier (Grenoble), 14 (1964), 1–32.

\bibitem{BurGulHus1}
  V.I. Burenkov, H.V. Guliyev, V.S. Guliyev,
  \textit{ Necessary and sufficient conditions for boundedness
  of the fractional maximal operators in the local Morrey-type spaces,}
   J. Comput. Appl. Math. {\bf 208} (1) (2007), 280-301.

\bibitem{BurGul2}
  V.I. Burenkov, V.S. Guliyev,
  \textit{ Necessary and sufficient conditions for
  the boundedness of the Riesz potential in local Morrey-type
  spaces,} Potential Anal. {\bf 30} (3) (2009), 211-249.

\bibitem{BurGogGulMus1} V. Burenkov, A. Gogatishvili, V.S. Guliyev, R. Mustafayev,
\textit {Boundedness of the fractional maximal operator in local Morrey-type spaces,}
Complex Var. Elliptic Equ. {\bf 55} (8-10) (2010), 739-758.

\bibitem{BurGogGulMus2} V. Burenkov, A. Gogatishvili, V.S. Guliyev, R. Mustafayev,
\textit {Boundedness of the Riesz potential in local Morrey-type spaces,}
Potential Anal. {\bf 35} (1) (2011), 67-87.

\bibitem{Chanillo} S. Chanillo,
\textit{ A note on commutators}, Indiana Univ. Math. J. 23 (1982), 7-16.

\bibitem{CarPickSorStep}
M. Carro, L. Pick, J. Soria, V.D. Stepanov,
\textit{ On embeddings between classical Lorentz spaces},
Math. Inequal. Appl. {\bf 4} (3) (2001), 397-428.

\bibitem{ChFra}
F. Chiarenza, M. Frasca,
\textit{ Morrey spaces and Hardy-Littlewood maximal
function}, Rend Mat. {\bf 7} (1987), 273-279.

\bibitem{ChFraL1}
F. Chiarenza, M. Frasca, P. Longo,
\textit{ Interior $W^{2,p}$-estimates for
nondivergence elliptic equations with discontinuous coefficients},
Ricerche Mat. {\bf 40} (1991), 149-168.

\bibitem{ChFraL2}
F. Chiarenza, M. Frasca, P. Longo,
\textit{ $W^{2,p}$-solvability of Dirichlet problem for
nondivergence elliptic equations with VMO coefficients},
Trans. Amer. Math. Soc. {\bf 336 } (1993), 841-853.


\bibitem{CRW}
  R. Coifman, R. Rochberg, G. Weiss,
  \textit{ Factorization theorems for Hardy spaces in several variables},
  Ann. of Math. {\bf 103} (2) (1976), 611-635.

\bibitem{Ding} Y. Ding,
\textit{Weak type bounds for a class of rough operators with power weights},
Proc. Amer. Math. Soc. 125 (1997), 2939-2942.

\bibitem{DingLu1} Y. Ding and S. Z. Lu,
\textit{Weighted norm inequalities for fractional integral operators with rough kernel},
Canad. J. Math. 50 (1998), 29-39.

\bibitem{DingYZ} Y. Ding, D. Yang, Z. Zhou,
\textit{Boundedness of sublinear operators and commutators on $L^{p,\omega}(\Rn)$},
Yokohama Math. J.  46 (1998), 15-27.

\bibitem{DingLu2}  Y. Ding and S. Z. Lu,
\textit{Higher order commutators for a class of rough operators},
Ark. Mat. 37 (1999), 33-44.

\bibitem{Duoandikoetxea} J. Duoandikoetxea, \textit{Fourier Analysis}, American Mathematical Society,
Providence, Rhode Island, 2000.

\bibitem{DY}
X.T. Duong, L.X. Yan, \textit{On commutators of fractional integrals},
Proc. Amer. Math. Soc. {\bf 132} (12) (2004), 3549-3557.

\bibitem{FazRag1}
G. Di Fazio, M.A. Ragusa,
\textit{ Commutators and Morrey spaces},
Boll. Un. Mat. Ital. 5A (7) (1991), 323-332.

\bibitem{FazRag2} G. Di Fazio, M.A. Ragusa,
\textit{ Interior estimates in Morrey spaces for
strong solutions to nondivergence form equations with discontinuous
coefficients}, J. Funct. Anal. 112 (1993), 241-256.

\bibitem{FazPalRag} G. Di Fazio, D. K. Palagachev and M. A. Ragusa,
\textit{Global Morrey regularity of strong solutions to the Dirichlet problem for elliptic equations with
discontinuous coefficients}, J. Funct. Anal, 166 (1999), 179-196.

\bibitem{Feicht} H. Feichtinger,
\textit{ An elementary approach to Wiener's third Tauberian theorem on Euclidean $n$-space},
Proceedings, Conference at Cortona 1984, Sympos. Math., 29, Academic Press 1987.

\bibitem{GarRub} J. Garcia-Cuerva and J.L. Rubio de Francia,
\textit{ Weighted Norm Inequalities and Related Topics},
North-Holland Math. 16, Amsterdam, 1985.

\bibitem{GulDoc}
 V.S. Guliyev, \textit{ Integral operators on function spaces on the homogeneous groups and
  on domains in $\Rn$}. Doctor's degree dissertation,
  Mat. Inst. Steklov, Moscow, 1994, 329 pp. (in Russian)

\bibitem{GulBook}
V.S. Guliyev, \textit{ Function spaces, Integral Operators and Two Weighted
  Inequalities on Homogeneous Groups. Some Applications},
  Cashioglu, Baku, 1999, 332 pp. (in Russian)

\bibitem{GulJIA} V.S. Guliyev,
\textit{ Boundedness of the maximal, potential and singular operators in the generalized Morrey spaces},
J. Inequal. Appl.  2009, Art. ID 503948, 20 pp.

\bibitem{GULAKShIEOT2012} V.S. Guliyev, S.S. Aliyev, T. Karaman, P. S. Shukurov,
\textit{Boundedness of sublinear operators and commutators on generalized Morrey Space},
Int. Eq. Op. Theory. 71 (3) (2011), pp. 327-355.

\bibitem{LuYang1} S.Z. Lu and D.C. Yang,
\textit{The central BMO spaces and Littlewood-Paley operators},
Approx. Theory Appl. (N.S.), 11 (1995), 72-94.

\bibitem{LLY}
G. Lu, S. Lu, D. Yang,
\textit{ Singular integrals and commutators on homogeneous groups},
Analysis Mathematica,  {\bf 28} (2002), 103-134.

\bibitem{LuWu} S.Z. Lu and Q. Wu,
\textit{ CBMO estimates for commutators and multilinear singular integrals},
Math. Nachr., 276 (2004), 75-88.

\bibitem{LuDingY}
S. Lu, Y. Ding, D. Yan,
\textit{ Singular integrals and related topics},
World Scientific Publishing, Singapore, 2006.

\bibitem{Miz}
T. Mizuhara, \textit{ Boundedness of some classical operators on
  generalized Morrey spaces},  Harmonic Analysis (S. Igari, Editor), ICM 90 Satellite
  Proceedings, Springer - Verlag, Tokyo (1991), 183-189.

\bibitem{Morrey} C.B. Morrey,
 \textit{ On the solutions of quasi-linear elliptic partial
  differential equations}, Trans. Amer. Math. Soc. {\bf 43} (1938), 126-166.

\bibitem{Nakai}
  E. Nakai, \textit{ Hardy--Littlewood maximal operator, singular integral
  operators and Riesz potentials on generalized Morrey spaces}, Math. Nachr.
  {\bf 166} (1994), 95-103.

\bibitem{Nakai1}
  E. Nakai, \textit{ A characterization of pointwise multipliers on the Morrey spaces,}
  Scientiae Mathematicae  {\bf 3} (2000), 445-454.

\bibitem{MuckWheed} B. Muckenhoupt and R. L. Wheeden,
\textit{Weighted norm inequalities for singular and fractional integrals},
Trans. Amer. Math. Soc, 161 (1971), 249-258.

\bibitem{Peetre} J. Peetre,
\textit{ On the theory of $M_{p,\lambda}$}, J. Funct. Anal. {\bf 4} (1969), 71-87.

\bibitem{St}
  Stein, E.M.: Singular integrals and differentiability of functions.
Princeton University Press, Princeton, NJ, 1970.

\bibitem{Stein58} E.M. Stein,
\textit{ On the functions of Littlewood-Paley, Lusin, and Marcinkiewicz},
Trans. Amer. Math. Soc. {\bf 88} (1958), 430-466.

\bibitem{Stein93} E.M. Stein,
\textit{ Harmonic Analysis: Real Variable Methods, Orthogonality and Oscillatory
Integrals}, Princeton Univ. Press, Princeton NJ, 1993.

\bibitem{Torch} A. Torchinsky,
\textit{ Real Variable Methods in Harmonic Analysis},
Pure and Applied Math. 123, Academic Press, New York, 1986.

\bibitem{TorWang}
A. Torchinsky and S. Wang, \textit{ A note on the Marcinkiewicz integral},
Colloq. Math. {\bf 60/61} (1990), 235-243.

\bibitem{Wiener1} N. Wiener,
\textit{ Generalized Harmonic Analysis}, Acta Math., 55 (1930), 117-258.

\bibitem{Wiener2} N. Wiener,
\textit{Tauberian theorems}, Ann. Math., 33 (1932), 1-100.

\end{thebibliography}
\end{document}